\documentclass[12pt]{amsart}

\DeclareFontFamily{U}{matha}{\hyphenchar\font45}
\DeclareFontShape{U}{matha}{m}{n}{
  <5> <6> <7> <8> <9> <10> gen * matha
  <10.95> matha10 <12> <14.4> <17.28> <20.74> <24.88> matha12
  }{}
\DeclareSymbolFont{matha}{U}{matha}{m}{n}
\DeclareFontFamily{U}{mathx}{\hyphenchar\font45}
\DeclareFontShape{U}{mathx}{m}{n}{
  <5> <6> <7> <8> <9> <10>
  <10.95> <12> <14.4> <17.28> <20.74> <24.88>
  mathx10
  }{}
\DeclareSymbolFont{mathx}{U}{mathx}{m}{n}

\DeclareMathSymbol{\obot}         {2}{matha}{"6B}
\DeclareMathSymbol{\bigobot}       {1}{mathx}{"CB}

\usepackage[pdfauthor={Congling Qiu}, 
    pdftitle={???}%
  dvips
  ]{hyperref}

\hypersetup{
    colorlinks=red,
    citecolor=red,
    filecolor=black,
    linkcolor=red,
    urlcolor=black
}

\usepackage[english]{babel}
\usepackage[toc,page]{appendix}

  \usepackage{multirow}

\usepackage[utf8]{inputenc}

\usepackage{xtab}
\usepackage{amsmath, amssymb
}
\usepackage{mathrsfs}
\usepackage[all]{xy}
\usepackage{extarrows}

\usepackage{enumerate}
\usepackage{mathtools,booktabs}
\usepackage{color}
\usepackage{cleveref}

\usepackage{epstopdf} 
\usepackage{booktabs}

\numberwithin{equation}{section}

\theoremstyle{plain}
\newtheorem{proposition}{Proposition}[section]
\newtheorem{conj}[proposition]{Conjecture}
\newtheorem{cor}[proposition]{Corollary}
\newtheorem{lem}[proposition]{Lemma}
\newtheorem{thm}[proposition]{Theorem}
\newtheorem{prop}[proposition]{Proposition}

\theoremstyle{definition}

\newtheorem{eg}[proposition]{Example}

\theoremstyle{remark}
\newtheorem{rmk}[proposition]{Remark}

\numberwithin{equation}{section}

  \newcommand{\BA}{{\mathbb {A}}} 
    \newcommand{\BC}{{\mathbb {C}}} 
     
     \newcommand{\BH}{{\mathbb {H}}}

    \newcommand{\BQ}{{\mathbb {Q}}} \newcommand{\BR}{{\mathbb {R}}}

     \newcommand{\BZ}{{\mathbb {Z}}}

    \newcommand{\cK}{{\mathcal {K}}} 
     
    \newcommand{\cO}{{\mathcal {O}}}

    \newcommand{\RM}{{\mathrm {M}}}

    \newcommand{\fc}{{\mathfrak{c}}} 
    \newcommand{\fe}{{\mathfrak{e}}}

    \newcommand{\fm}{{\mathfrak{m}}} 
     \newcommand{\fp}{{\mathfrak{p}}}

    \newcommand{\ol}{\overline}

    \newcommand{\pair}[1]{\langle {#1} \rangle}

    \newcommand{\incl}{\hookrightarrow}

    \newcommand{\bsl}{\backslash}

  \newcommand{\ep}{\epsilon}
  \newcommand{\vpl}{\varprojlim}
             
 \newcommand{\vil}{\varinjlim}  
 \newcommand{\lb}{\left(} \newcommand{\rb}{\right)}

    \newcommand{\ad}{{\mathrm{ad}}}

    \newcommand{\Ch}{{\mathrm{Ch}}}

    \newcommand{\End}{{\mathrm{End}}}

     \newcommand{\GL}{{\mathrm{GL}}}

    \newcommand{\Hom}{{\mathrm{Hom}}}

    \newcommand{\Mor}{{\mathrm{Mor}}}

    \newcommand{\ord}{{\mathrm{ord}}} 
    \newcommand{\PGL}{{\mathrm{PGL}}}

\newcommand{\rf}{{\mathrm{f}}}

    \newcommand{\sym}{{\mathrm{sym}}}

      \newcommand{\pr}{{\mathfrak{pr}}}

\newcommand\supervisor[1]{\def\@supervisor{#1}}

\newcounter{elno}

\renewcommand{\cong}{\simeq}

 \author{Congling Qiu}
 
\begin{document}
   \title{Hyperelliptic Shimura curves and
 $L$-functions of  central vanishing order at least 3}
 \begin{abstract} 

 We use hyperelliptic Shimura curves to find
  triple product $L$-functions  of   Hilbert newforms  
with 
   central vanishing orders proved to be at least 3.  
  
 \end{abstract} 
 \maketitle

\section{Introduction}

In \cite{GZ,Zag},  Gross and Zagier proved that the $L$-function  of  a  certain modular elliptic curve over $\BQ$  
 has central vanishing order    3. 
This was essential in Goldfeld's  effective solution to the Gauss class number problem for imaginary quadratic fields  \cite{Gold}.
According to folklore conjectures,  such $L$-functions with higher central vanishing orders are very rare.  Moreover, such vanishing can not be proved by numerically computation. Thus, Gross and Zagier's example  is remarkable. 
In this paper, we present some analogs of Gross and Zagier's example in triple product $L$-functions of abelian varieties.

 \begin{thm}
\label{thmintro} 

 Let   $F=\BQ(\zeta_{32}+\zeta_{32}^{-1})$, the maximal totally real subfield of the cyclotomic field of 32nd roots of unity.
Let $\fp$ be the unique prime of $F$ above 2. 
Let $A_1,A_2,A_3$ be pairwise non-isomorphic 4-dimensional simple  modular abelian varieties of   $\GL_2$-type over $F$   and of conductor $\fp$.
 Then  
$$\ord_{s=2}L\lb s, h^1(A_1)\otimes  h^1(A_2)\otimes  h^1(A_3) \rb\geq 3\cdot 4^3.$$
\end{thm}  

Indeed, 
there are exactly four  modular abelian varieties of $\GL_2$-type over $F$  of  conductor $\fp$ and dimension 4, which are all simple.
They exactly 
correspond to the only 4 Galois orbits of Hilbert newforms over $F$ with Atkin--Lehner sign $-1$.  And each Galois orbit  is of size 4.
Then the fact that the triple product $L$-function     has a holomorphic continuation to $s\in\BC$   follows from \cite{Gar,PSR}.   

We actually prove a slightly more general version of   \Cref{thmintro} in terms of 
Hilbert newforms over $F$.
   In \Cref{exmain}, we show that 
for any 3  pairwise distinct  newforms of the 16  Hilbert newforms (in the four Galois orbits of size 4) above, their triple product $L$-function has
central vanishing order   at least 3.

In general, when  some of the 3 newforms are the same, the triple product $L$-function is decomposable. See \Cref{sec:6} . In this case, using usual modular forms over $\BQ$,
we  get many degree 6  $L$-functions with central vanishing order   at least 2 (\Cref{thm:ord2}).

Our proof of \Cref{thmintro} 
 is similar to the one of Gross and Zagier in spirit.  They used their celebrated formula on modular curves,
 which  relates  $L$-functions of modular forms to the heights of (Heegner) CM 0-cycles.
 We use an analog proved by   Yuan, S. Zhang, and W. Zhang \cite{YZZ0} for some Shimura curves, which  relates  triples product $L$-functions to the heights of modified diagonal cycles on the triple product of Shimura curves. Both of us show that our cycles accidentally  vanish for some reasons. 
 
 In our case,   our cycles vanish due to an exceptional hyperelliptic involution,  which is the work of Demb\'el\'e \cite{Dem}. See \Cref{Demcurve}.
 Here, an involution  is  exceptional if it does  not come from an Atkin--Lehner involution. (In general, if an Atkin--Lehner involution is a hyperelliptic involution, it is not useful to produce higher central vanishing of  $L$-functions. See \Cref{final remark}.)
  Interestingly, the specific modular curve used for Gross and Zagier's order 3 example, which is $X_0(37)$,  is of genus 2. And its hyperelliptic involution is  exceptional.  
It might be a coincidence, since the hyperellipticity   is essential for our result, but it is not  used in  \cite{GZ,Zag}. 

 Finally, we want to remark that  there are only finitely many  hyperelliptic Shimura curve, as shown in \cite{Qiufin}, let along the ones satisfying our specific criterions \Cref{sec5}. 
 Moreover, we are mostly interested in the ones of genus at least 3 due to the decomposability of  
  triple product $L$-functions mentioned above.   They are even  rarer. For example, there are no such  modular curves as we will see in \Cref{sec:genus2}.
   Such rareness  is also a 
  facet of the expectation
that $L$-functions  of central vanishing order at least 3 are very rare.

   \subsection*{Structure of the paper}

In \Cref{sec:1} and \Cref{sec:2}, we review some background on abelian varieties and Shimura curves, respectively. 
In \Cref{sec4}, we relate modified diagonals to the hyperellipticity of Shimura curves. 
In \Cref{sec5}, we give criteria for triple product $L$-functions to have central vanishing order at least 3. 
In \Cref{sec:6}, we consider the cases where the triple product $L$-function is decomposable. 
In \Cref{sec:genus2}, we study examples of hyperelliptic modular curves and obtain many degree-6 $L$-functions over $\BQ$ with central vanishing order at least 2. 
Finally, in \Cref{Demcurve}, we study the hyperelliptic Shimura curve of Demb\'el\'e and prove \Cref{thmintro}.

  \subsection*{Acknowledgments}

The author would like to thank Peter~Sarnak for raising the question of possible analogs of Gross--Zagier's example in higher dimensions during a talk by the author at Princeton.
The author is  grateful to  Yifeng~Liu and Wei~Zhang for many helpful discussions   in the early stages of this work.
 The author also thanks Benedict~Gross and Adam~Logan for valuable communications that helped correct a mistake in the first version of this paper.
  
  \section{Triple product of abelian motives}  \label{sec:1}

 For a smooth projective  variety $V$ over a  field,   let $\Ch_i(V)$ be the Chow group  of $V$   of   dimension $i$ cycles with   $\BQ$-coefficients, modulo rational equivalence.  For a morphism $f:V\to U$,  let  $ \Gamma_f\in \Ch_*(V\times U)$ be  the graph of $f$.

   Let $A$ be an abelian variety  over a  field $F$ of dimension $g$, with dual abelian variety $A^\vee$. Let $\pr \in \Ch_g(A\times A)$  be the $(2g-1)$-th Deninger--Murre projector, that is denoted by $\pi_{2g-1}$ in  \cite[Theorem 3.1]{DM}.
  For a variety $X$ over $F$, 
the abelian group structure of $A$ induces an abelian  group structure on $  \Mor(X,A) $, the set of morphisms of $F$-varieties. 
Then by regarding $X\times A$ as an abelian scheme over $X$, \cite[Corollary A.3.2]{QZ2}  implies  that the following map  is a homomorphism of abelian groups: 
\begin{align*}
\Mor (X,A)& \to  \Ch_{\dim X}(X\times A),\\
 f&\mapsto {\pr}\circ [\Gamma_f].
 \end{align*}  
  Denote  the unique  $\BC$-linear extension $ \Mor (X,A)_\BC \to  \Ch_{\dim X}(X\times A)$ by $$f\mapsto {\pr_f}.$$

 For  an idempotent $\iota \in \End(A)_\BC$, by \cite[Proposition 3.3]{DM},
 $\pr_\iota$  is a projector.
  Moreover, if $\iota$ is in the center of $\End(A)_\BC$, one can easily check that its image  $\iota^\dag$
 under a  Rosati involution  
  on $\End(A)_\BC$ associated  with a polarization does not depend on the choice of the polarization.

 \begin{prop} \label{prop:van0}  
 Let the symmetric group $S_n$ act on $A^n$ by permuting the components. 
   Assume that 
  $ \dim \iota _*H^1(A)<n$.
For $z \in \Ch_*(A^n)$ invariant by $S_n$, 
we have  $ (\pr_\iota^{ n})_* z  =0  $.  
     
 \end{prop}   
   \begin{proof}
  The proof is the same as \cite[Corollary 2.2.3]{QiuFP}.
\end{proof}

 \begin{prop} \label{prop:van1}   
    Let $f:X\to A$ be a morphism from a smooth projective curve over $F$ whose associated Albanese morphism to $A$ is surjective.
    Let $\Delta\subset X^2$ be the diagonal.
    Assume  that $\iota$ is in the center of $\End(A)_\BC$ such that
     $\iota^\dag =\iota$
and
$ \dim  \iota _*H^1(A)\leq 2$.
    Then $ (\pr_\iota^{ 2})_*\Ch_*(A^2) =\BC\cdot (\pr_\iota^{ 2})_*(f^2)_*[\Delta]$, which is nonzero.

 \end{prop} 
 \begin{proof}
   The proof   is the same as \cite[Corollary 2.2.7 (1)]{QiuFP}.
\end{proof}  
\begin{rmk}
The relation between the above two propositions and \cite[Corollary 2.2.3 , Corollary 2.2.7 (1)]{QiuFP} is not just similar. Indeed, if $A$ is simple,  
 we have an isomoprhism between the motives $(A,\pr_\iota)$ and 
 $(X,\delta_f)$ defined in \cite[Section 2.4]{QiuFP}\footnote{In loc. cit., we used $\phi$ instead of $f$ to denote a morphism from $
X$ to $A$. }. This isomoprhism can be easily proved using the decomposition of $\delta_f$ 
in  (the summand on the right hand side of)  \cite[(2.13)]{QZ2} and the nilpotency  of Kimura \cite[Proposition 7.5]{Kim}. 
\end{rmk}

   \section{Shimura curves}\label{sec:2}
 
 Let $F$ be a totally real number field and $\BA_\rf$  its ring of finite adeles.
  Let $B$  be a   quaternion algebra   split   at one infinite  place $\tau:F\incl \BR$ of $F$ and division at all other infinite places of $F$.
 Let $B^+\subset B^\times$ be the subgroup of elements with totally positive norms, equivalently, with positive determinants  in $B_\tau^\times\cong \GL_2(\BR)$.  
 Let  $\BH$
 be the complex upper half-plane, and  let $B^+\subset\GL_2(\BR)$  act on $\BH$ by the fractional linear transformation. 

  For an open compact-modulo-center subgroup $U$ of $B^\times(\BA_\rf)$, we have the smooth   compactified   Shimura curve $ X_U$ for  $B^\times$ of level $U$ over $F$ \cite[3.1.4]{YZZ}\footnote{Technically, sometimes by   ``Shimura curve $ X_U$ for  $B^\times$", it is required that $U$ is compact. And if $U$ contains the center of $B^\times(\BA_\rf)$, $X_U$ should be refered to a Shimura curve for $PB^\times$. We do not make this distinction here.}. 
  The  complex uniformization  of  $X_U$ via the distinguished infinite place $\tau: F\incl \BR\subset \BC$ is given  by
\begin{equation}X_{U,\tau,\BC} \cong B^+ \bsl \BH \times B^\times(\BA_\rf)/U\coprod \{\text{cusps}\}.\label{Xcomplex}\end{equation}
Here,
 the cusps exist if and only if $X_U $ is a modular curve, i.e. $F=\BQ$, and $B $ is the matrix algebra.  

 Consider $\vpl _U X_U$ where  the transition morphisms are   the natural morphisms of Shimura curves. 
 There is a natural   right action of $B^\times(\BA_\rf)$  on $\vpl _U X_U$. At the level of the  complex points, the action is by right multiplication via  \eqref{Xcomplex}.

For $x \in X_U(\BC)$ that is the image of a point in $\BH \times G(\BA_\rf)/U$ 
with a nontrivial stabilizer in $B^+$,  we call $x$ an elliptic point and let $r_x$ be  the order of the stabilizer. There are only finitely many elliptic points on $X_U$.
Let $$\fe_U = \sum_{\text{elliptic } x } \lb 1-\frac{1}{r_x}\rb x $$
Let $K_U$ be a canonical divisor of $X_U$, and 
$\cK_U = K_U+\fe_U$
an orbifold canonical divisor of $X_U$.
Let $\fc_U$ be the divisor of the cusps of $X_U$, which is 0 unless $X_U$ is a modular curve.
Let
$$L_U = \cK_U  +\fc_U.$$ 
 Then $\deg L_U>0$
 and $[L_U]$
 is compatible under pullback by  the natural morphisms of Shimura curves as the level  $U$ changes. 
 And thus defines a divisor class
  on  $\vpl _U X_U$. This divisor class is invarant  by the  right action of $B^\times(\BA_\rf)$.  
 Let $$\lambda_U = [L_U]/\deg L_U^\circ$$ where   $L_U^\circ$ is the restriction of $L_U$ to 
 any geometrically connected component of $X_U$.
 In \cite[Remark 3.2.4]{QZ2}, we further noticed that if $X_U$ is a modular curve, then  
$\fc_U$ is a multiple of $\lambda_U$. So  
\begin{equation}\label{lambda}
  \lambda_U = [\cK_U]/\deg \cK_U^\circ,
\end{equation}
provided that $\deg \cK_U>0$.

 For an abelian variety  $A$   over $F$, let  $\Mor^0(X_U,A)$ be the subspace of  $f\in \Hom(X_U,A)_\BQ $  such that
for some (equivalently for every) representative $\sum_i a_i p_i$ of $\lambda_{U,\ol F}$, where $a_i\in \BQ$ and $p_i\in X_U( \ol F)$, the sum of  $ f(p_i)\otimes a_i $ over $i$'s in the group $A(\ol k)\otimes {\BQ}$ is the identity of $A$.
  Let 
 $$
\pi_A :=\vil_U\Mor^0(X_U,A),
 $$
where  the transition morphisms are  induced by the natural morphisms of Shimura curves.
  The $\BQ$-vector space $\pi_A$ carries a natural left action by $B^\times(\BA_\rf)$, induced by the right action of $B^\times(\BA_\rf)$ on $\vpl X_U$.
Then  the subspace of $U$-invariants is $$\pi_A^U=\Mor^0(X_U,A).$$
 Let $\End(A)_\BQ$   act on $\pi_A^U$ and $\pi_A$ from left by post-composition, and one may regard $\pi_A$ as a $B^\times(\BA_\rf)$-representation with $\End(A)_\BQ$-coefficients.

  \begin{thm}[{\cite[3.2.2, 3.2.3]{YZZ}}]\label{YZZ322} 
  Let $A$ be a simple abelian variety over $F$ such that 
 $
 \pi_A\neq 0$. 
 
(1) The $\BQ$-algebra $\End(A)_\BQ$ is a finite field extension of $\BQ$ and $[\End(A)_\BQ:\BQ]=\dim A$. 

(2)   Consider the set of pairs 
$ (A,\iota)$ where $A$ is a simple abelian variety over $F$ such that 
 $
 \pi_A\neq 0$ and  $\iota: \End(A)_\BQ\incl \BC$ is an embedding. The map $(A,\iota)\to \pi_{A,\iota}$
 is a bijection to  the set of finite  components of 
   automorphic representations of $B^\times$, whose Jacquet--Langlands correspondence to $\GL_{2,F}$ is cuspidal and holomorphic of weight 2.

 \end{thm} 

 \begin{rmk}\label{rmkMA}
 (1)
By (1) of the theorem, $A$ is of $\GL_2$-type.  A simple abelian variety over $F$ of $\GL_2$-type with 
 $\pi_A\neq 0$ for some $B$ as above is called  modular (or automorphic in \cite{YZZ}).

 (2) 
 Since $\End(A)_\BQ $  is a field so that 
 irreducible idempontents of $\End(A)_\BQ $ are in bijection with  embeddings 
 $  M_A\incl \BC$.  Moreover, they are automatically central. Thus, here we use the same notation $\iota$ as in \Cref{prop:van0}   and 
 \Cref{prop:van1}.
 \end{rmk}

Below we simplify notations and let $\lambda, \fc, \fe$ denote
 $\lambda_U, \fc_U, \fe_U$ if $U$ is clear from the context.
 Assume $\deg(K_U)>0$. 
 Let  $$\xi = [K_U]/\deg K_U^\circ.$$
By definition, 
      $\pr_{f,*}\lambda = 0$. Then by \eqref{lambda},
we have  
\begin{equation}\label{xie2}
  \pr_{f,*}\xi = 0\text{ if and only if }\pr_{f,*}[\fe] = 0.
\end{equation} 

 \section{Diagonals and hyperellipticity}\label{sec4}

Consider  the $n$-th power  of  the Shimura curve $ X_U$.
 Let  $$\Ch(n)=\vpl_U  \Ch_*(X_{U}^n) ,$$
  where  the transition maps in the inverse limit are pushforwrds   by the natural morphisms between Shimura curves. 
Let $ B^{\times,n}(\BA_\rf)$ act on $\ol\Ch(n)$ (from left) by   pullback.
Let  $\pi_i$  be of the form $\pi_{A_i,\iota_i} $ for some pair $(A_i,\iota_i)$ as in \Cref{YZZ322} (2), which is a representation of  $ B^{\times}(\BA_\rf)$.  
   We have a linear map
   $   \pi_{i}^U\to
       \Ch_{1}(X_U\times A_i)$, sending $f_i$ to  $ \pr_{f_i}$. Let
       \begin{align*}
\boxtimes_{i=1}^n\pi_{i}^U& \to  \otimes_{i=1}^n  \Ch_{1}(X_U\times A_i),\\
 f&\mapsto \pr_f.
 \end{align*}  
        be the tensor product of these maps.

The following is a slight generalization of \cite[Theorem 3.3.2]{QZ2}. The proof is the same.
 \begin{thm}\label{general}

 Let $ H\subset  B^{\times,n}(\BA_\rf)$ be a subgroup. Let $\Pi\subset \boxtimes_{i=1}^n\pi_{i}$ be a subspace be stable by $H$. If 
 $\Pi$  has no nonzero  $H$-invariant linear forms, then for $z=(z_U)_U\in\Ch(n) $ invariant by $H$, we have $$ \pr_{f,*} z = 0 $$ for  all $f\in \Pi $.
\end{thm}
In particular, the theorem applies to a sub-$B^{\times,n}(\BA_\rf)$-representation $\Pi$ of $\boxtimes_{i=1}^n\pi_{i}$.

Let $\Delta_n = \Delta_{U,n}$ be the diagonal $X_U^n$.
      Let 
  $$[\Delta_n]_{f} =  \pr_{f,*}[\Delta_n] $$  in $ \Ch_1(A_1\times A_2\times...  \times A_n) .$

Accordingly, we have the following   slight generalization of  \cite[Corollary 3.3.4]{QZ2}.
 
 \begin{cor}\label{generalcor}

If a sub-$B^{\times,n}(\BA_\rf)$-representation $\Pi$ of $\boxtimes_{i=1}^n\pi_{i}$.
   has no nonzero diagonal-$B^\times(\BA_\rf)$-invariant linear forms, then   $$[\Delta_n]_{f}= 0 $$ for all  $f\in \Pi$.
\end{cor}
\begin{rmk}
In the case $n=3$ and  $\Pi=\boxtimes_{i=1}^n\pi_{i}$,   this  condition has been extensively studied in \cite{Qiufin}.
\end{rmk}

     Here are some  easy examples for applying this corollary to $\Pi=\boxtimes_{i=1}^n\pi_{i}$.

      \begin{eg} \label{eg:van}
 
 (1) If $n=1$, let $H= B^{\times}(\BA_\rf)$ and we have $[\Delta_1]_{f_1}  = 0 $.
 
 (2) If $n=2$, $H$ is the diagonal $B^{\times}(\BA_\rf)$, and $z_U =[\Delta_2]  $, then $[\Delta_2]_{f_1\otimes f_2}  = 0 $
 if $\pi_1\not\cong \pi_2^\vee$.
 
 (3) If the product of the central characters of $\pi_i$'s is not trivial, then  $[\Delta_n]_{f_1\otimes f_2,...\otimes f_n}= 0 $.
 \end{eg}
  For a subset $I$ of $\{1,2,...,n\}$, let   $p_I$ be the projection of $X_U^n$ to the product of the components indexed by $I$.

\begin{prop} \label{prop:diag_van}

(1) Assume   every  geometrically connected component of $X_{U}$ is hyperelliptic.  
 Then $$[\Delta_3]_{f_1\otimes f_2\otimes f_3}   =  -\sum_{\substack{I \coprod J = \{1,2,3\} \\ I = \{a,b\}, J=\{c\}}}p_I^* [\Delta_{2}]_{f_a\otimes f_b} \cdot p_J^* \pr_{f_c,*}\xi 
    $$
for  all $f_i\in \pi_i^U$.  

(2) Further  assume   $\pi_3\not\cong \pi_1^\vee$ and  $\pi_3\not\cong \pi_2^\vee$.
Then $$[\Delta_3]_{f_1\otimes f_2\otimes f_3}  =  -p_{1,2}^* [\Delta_{2}]_{f_1\otimes f_2} \cdot p_{3}^* \pr_{f_3,*}\xi ,$$
for  all $f_i\in \pi_{i}^U$.

(3) Further   assume $\pi_i\not\cong \pi_j^\vee$ for any $i\neq j$ in $\{1,2,3\}$.
Then   $$[\Delta_3]_{f_1\otimes f_2\otimes f_3}   = 0 $$
for  all $f_i\in \pi_{i}^U$.
 
  \end{prop}    
  \begin{proof}
  The modified diagonal defined by Gross and Schoen \cite{GS,Zha10} is   the algebraic cycle
\begin{align}\label{eq:mod diag}
[\Delta_3]_{\xi} = [\Delta_3] - \sum_{\substack{I \coprod J = \{1,2,3\} \\ |I| = 2}}p_I^* [\Delta_{2}] \cdot p_J^* \xi + \sum_{\substack{I \coprod J = \{1,2,3\} \\ |I| = 1}}p_I^* [\Delta_{1}] \cdot p_J^*(\xi \times \xi)
\end{align}
 By hyperellipticity and  \cite[4.5]{GS},  $[\Delta_{\xi}] =0$.
 The proposition follows from  \Cref{eg:van} (1)(2).
  \end{proof}

Combining (1) and (2) of the proposition  with \eqref{xie2}, we have a corollary.

  \begin{cor}\label{corCM}  (1) Assume   every  geometrically connected component of $X_{U}$ is hyperelliptic.  
  Then $[\Delta_3]_{f_1\otimes f_2\otimes f_3}  =  0$
for  $f_i\in \pi_{i}^U$ such that $\pr_{f_i,*}[\fe] = 0$.

(2) Further  assume   $\pi_3\not\cong \pi_1^\vee$ and  $\pi_3\not\cong \pi_2^\vee$.
Then $[\Delta_3]_{f_1\otimes f_2\otimes f_3}  =  0$
for $f_i\in \pi_{i}^U$  such that $\pr_{f_3,*}[\fe] = 0.$

  \end{cor}

 \begin{rmk}  
  
  Note that  elliptic points are CM points and  $\fe$ is a sum of CM cycles. 
Then  the equation  $\pr_{f,*}[\fe] = 0$, if holds,  can usually be verified   by local conditions.  See \Cref{sec:genus2}.

   \end{rmk}

  \section{Triple product $L$-functions of central vanishing order at least 3} \label{sec5}
  
   We keep the notations from the last section, except that
from now on, the notation $\pi$ or $\pi_i$ will denote an  automorphic  representation 
of 
$B^\times $, instead of its finite part. We always assume that 
  it is  cuspidal holomorphic of weight $2$  in the sense that
its Jacquet--Langlands correspondence   to $\GL_{2,F}$  is   a   cuspidal automorphic representation  holomorphic of weight $2$.

  For a  place $v$ of $F$,  let  $\ep(B_v)=1$ if is $B_v$ is split,  and 
  let  $\ep(B_v)=-1  $ if is $B_v$ is   division. So if $v=\tau$, 
   $\ep(B_v)=1$ and for any othet infinite place $v$,    $\ep(B_v)=-1.$ Then
   \begin{equation*}\label{epB}
    \prod_v \ep(B_v) = 1.
   \end{equation*}
 Here the product   is   over all places of $F$.
  Let $\ep(1/2,\pi_{1,v}\otimes\pi_{2,v}\otimes\pi_{3,v})$ be the local root number  of $\pi_{1,v}\otimes\pi_{2,v}\otimes\pi_{3,v}$.
A result of Prasad \cite{Pra,Pra1} states the following.

 \begin{thm}\label{thm:Pra} If 
  $$\Hom_{B_v^\times}( \pi_{1,v}\otimes\pi_{2,v}\otimes\pi_{3,v},\BC )\neq\{0\},$$
then it has dimension 1 and 
$\ep(1/2,\pi_{1,v}\otimes\pi_{2,v}\otimes\pi_{3,v})=\ep(B_v)$.

\end{thm} 

   \begin{cor} 
    \label{corsign}

If 
  $$\Hom_{B_v^\times}( \pi_{1,v}\otimes\pi_{2,v}\otimes\pi_{3,v},\BC )\neq\{0\}$$
  for all finite places $v$ of $F$,
then $\ep(1/2,\pi_{1}\otimes\pi_{2}\otimes\pi_{3})=-1$.
 
\end{cor} 
   \begin{proof}    At  every infinite place $v$,   $\ep(1/2,\pi_{1,v}\otimes\pi_{2,v}\otimes\pi_{3,v})=-1$ (see \cite[Section 9]{Pra}).    At  every infinite place $v$, $\ep(B_v)=-1$ for except one place. So 
by \Cref{thm:Pra}, we have
$$\ep(1/2,\pi_{1}\otimes\pi_{2}\otimes\pi_{3})=\prod_v\ep(1/2,\pi_{1,v}\otimes\pi_{2,v}\otimes\pi_{3,v})=-\prod_v \ep(B_v)=-1.$$
  \end{proof} 
  
   Let $S_B$ be the finite set of \textit{finite} places $v$ of $F$ such that $B_{v} $ is a division quaternion algebra.  
  Below we only consider  a maximal open compact subgroup   $U =\prod\limits_{v\text{ finite}} U_v \subset PB^\times(\BA_\rf)$ until the final remark (\Cref{final remark}).
In other words,
its   $v$-component is 
          \begin{equation}\label{Kv}
U_v= PB_{v }^\times \text{ for }v\in S_B,
 \end{equation}
 and 
             \begin{equation}\label{Kv0}
U_v=F_v^\times \GL_2\lb \cO_{F_v}\rb   \text{ or }  F_v^\times \pair{\Gamma_0\lb \fm_{F_v}\rb,  g_v}
\text{ for } v\not
      \in S_B,
 \end{equation}
 where $\fm_{F_v}$ is the maximal ideal of $\cO_{F_v}$, $\Gamma_0\lb \fm_{F_v}\rb$  is the subgroup  of matrices in
 $ \GL_2\lb \cO_{F_v}\rb  $ that are upper triangular modulo $\fm_{F_v}$ 
 and
$
                   g_v\in \begin{bmatrix}
0 & 1 \\
\cO_{F_v} \bsl \fm_{F_v}  & 0
\end{bmatrix}
$.
If $\pi^U\neq 0$, then $\pi$ has trivial central character and thus self-dual, and
  $\dim \pi_v^{U_v}=1$ for $v$ finite.
Explicitly, for $v\in S_B$, $\pi_v $ is the trivial representation of $B_{v }^\times$. Its Jacquet-Langlands correspondence to $\GL_2(F_v)$   is the
the Steinberg representation. 
In this cases, the root number $\ep(1/2,\pi_v)  = -1$.  
For $v\not \in S_B$, if
$U_v=F_v^\times \GL_2\lb \cO_{F_v}\rb$ for $ v\not
      \in S_B$, then $\pi_v$ is an unramified representation of  $\GL_2(F_v)$.  
If $U_v=F_v^\times \pair{\Gamma_0\lb \fm_v\rb,  g_v}$,  then 
 $\pi_v$  is the
the Steinberg representation twisted by the unique nontrivial unramified quadratic character of $F_v^\times$. 
In both cases, the root number $\ep(1/2,\pi_v)  = 1$.  

This lemma is not needed for the results in this  section. We will use it later.
\begin{lem} \label{lemroot} 
If $\pi^U\neq 0$,  then
$\ep(1/2,\pi)  = -1$.
 \end{lem}
\begin{proof} 
   At an infinite place, 
$\ep(1/2,\pi_v)  = -1$.  At  every infinite place $v$, $\ep(B_v)=-1$ for except one place.  And at  every infinite place $v$, as we have seen above the lemma, 
  $\ep(1/2,\pi_v) = \ep(B_v)$.
So   we have the first of the following equations
$\ep(1/2,\pi)  = -\prod_v \ep(B_v)=-1.$ 
\end{proof}

   This lemma is  essential for the results in this  section.  
       \begin{lem} \label{lemrootform} 
  If $\pi_i^U\neq 0$ for $i=1,2,3$ and $v$ finite,  then $$\Hom_{B_v^\times}( \pi_{1,v}\otimes\pi_{2,v}\otimes\pi_{3,v},\BC )\neq\{0\}.$$
For $\ell\neq0$ in it and $\phi_i\neq 0\in\pi_{i,v}^{U_v}$ (both unique up to scalar as  seen above), we  have 
\begin{equation}
  \label{lneq0}
 \ell(\phi_1\otimes\phi_2\otimes\phi_3)\neq0. 
\end{equation}
 \end{lem} 
   \begin{proof} 
   If $v\in S_B$ so that $\pi_v=1$, the lemma is trivial. Otherwise, the lemma is \cite[Proposition 6.1, 6.3]{GP}.
  \end{proof}

Then by \Cref{corsign}, we have the following
   \begin{cor} \label{corsign2}

If $\pi_i^U\neq 0$ for $i=1,2,3$,  then  $\ep(1/2,\pi_1\otimes\pi_2\otimes\pi_3)=-1$.
\end{cor} 
Let $\pi_{i,\rf}$ denote the finite component of $\pi_i$. Then $\pi_{i,\rf}  = \pi_{A_i,\iota_i}$ for some pair $(A_i,\iota_i)$ as in \Cref{YZZ322} (2). 

  The main theorem of this section is the following.
 \begin{thm}\label{thmmain}
Assume \begin{itemize}
  \item $U$ satisfies \eqref{Kv} and \eqref{Kv0},
  \item  every  geometrically connected component of $X_{U}$ is hyperelliptic, and 
\item $\pi_i^U\neq 0$ for $i=1,2,3$.
\end{itemize} 
If   $\pi_i$'s are pairwise non-isomorphic,
then $$\ord_{s =1/2} L(s,\pi_1\otimes\pi_2\otimes\pi_3)\geq 3. $$ 
\end{thm} 

\begin{proof}
  The generalized Gross--Kudla conjecture \cite{GK} is proved  in \cite{YZZ0}  under the first two assumptions in
the theorem. 
By this and  \eqref{lneq0}, 
the central derivative $L'(1/2, \pi_1\otimes\pi_2\otimes\pi_3)$ is a multiple of the height of 
$[\Delta_3]_{f_1\otimes f_2\otimes f_3}  $, where  $f_i\neq 0\in \pi_{i,\rf}^U$, unique up to scalar. 
  But by 
\Cref
{prop:diag_van} (3),   $[\Delta_3]_{f_1\otimes f_2\otimes f_3}  =0$. 
Now the theorem follows from \Cref{corsign2}.
\end{proof}
We will use this theorem to prove \Cref{thmintro}.

Replacing \Cref
{prop:diag_van} (3) by \Cref
{corCM}  (2), we have the following
   more general theorem.
 \begin{thm}\label{thmmainCM}
Assume \begin{itemize}
  \item $U$ satisfies \eqref{Kv} and \eqref{Kv0},
  \item  every  geometrically connected component of $X_{U}$ is hyperelliptic, and 
\item $\pi_i^U\neq 0$ for $i=1,2,3$,
\item   $\pi_3\not\cong \pi_1 $,  $\pi_3\not\cong \pi_2 $.
\end{itemize} 
Let $f_3\neq 0\in \pi_{3,\rf}^U$, unique up to scalar. 
If     $\pr_{f_3,*}[\fe] = 0$, 
then $$\ord_{s =1/2} L(s,\pi_1\otimes\pi_2\otimes\pi_3)\geq 3. $$ 
\end{thm}  
  When  $ \pi_1\cong \pi_2\cong \pi_3$, $L(s,\pi_1\otimes\pi_2\otimes\pi_3)  $ has  central vanishing order at least 3 automatically. 
  We discuss this in the coming section.

 \begin{rmk}\label{final remark}
 (1) The condition we impose on the levels of $\pi_i$'s, i.e., on $U$, above is for the purpose of getting \Cref{lemrootform}, especially \eqref{lneq0}.
 Slightly more generally,  we may further allow $U_v$ to be   $F_v^\times  \cO_{B_v}^\times$    for $v\in S_B$ and the maximal order $  \cO_{v} $ of $B_v$,
 and allow $U_v$ to be
             $  F_v^\times  \Gamma_0\lb \fm_{F_v}\rb $
 for $ v\not
      \in S_B$. Then       we need to match the representations so that the lemma still holds. For example, for $v\in S_B$, let $\pm1$ be the two  representations
      of  $B_v^\times/F_v^\times O_v^\times\cong \BZ/2\BZ $. Then  
   we need  the product of    $\pi_{i,v}$'s to be 1. In particular, if the Atkin--Lehner involution, defined by the generator $B_v^\times/F_v^\times O_v^\times$, is hyperelliptic, then 
   all $\pi_{i,v}$'s are $-1$. And  the lemma  will not hold.  
      
   (2)   For $ v\not
      \in S_B$, more generally, the Atkin--Lehner involution is defined  if 
 $U_v=   F_v^\times \Gamma_0\lb \fm_{F_v}^{n_v}\rb$ for some positive integer $n_v$.  Indeed, any $                   g_v\in \begin{bmatrix}
0 & 1 \\
 \fm_{F_v} ^{n-1} \bsl \fm_{F_v}^n  & 0
\end{bmatrix}
$ normalizes $U_v$ and defines   the same Atkin--Lehner involution $w_v$ on $X_U$ at $v$.
If         $\dim \pi_{i,v}^{  U_v} = 1 
$, then  $w_v$  acts on $ \pi_{i,v}^{  U_v} $ by $\pm1$. 
If the action is by $-1$
 and $f_i \neq 0  \in \pi_{i,\rf}^U$,   one can use 
  $w_v$ to show that
    $[\Delta_3]_{f_1\otimes f_2\otimes f_3}     =0$. Indeed,  $(w_v,  w_v,  w_v)$ fixes 
     $[\Delta_3]$ and thus fixes  $[\Delta_3]_{f_1\otimes f_2\otimes f_3}  $, but sends each $f_i$ to $-f_i$.
    The case when $w_v$  is a   hyperelliptic involution is simply the case that any $\pi$ with $ \pi_{i,\rf}^U\neq0$ satisfies that 
    $\dim \pi_{i,v}^{  U_v} = 1 $ and $w_v$  acts on $ \pi_{i,v}^{  U_v} $ by $-1$.      However, in this case, \eqref{lneq0} fails by the same argument. Thus  $w_v$ being a   hyperelliptic involution  can not be used to  produce higher central vanishing of  $L$-functions. 
   
   (3) Finally, the above remarks on an Atkin--Lehner involution at a single finite place extend to a product of Atkin--Lehner involutions at several finite places in an obvious way.

\end{rmk}

   \section{Decomposable triple product $L$-functions}  \label{sec:6}  

With the generalized Gross--Kudla conjecture\cite{GK,YZZ0},
we have used the (modified) diagonal cycle to study triple product $L$-functions. Now we take the opposite direction, and want to understand the behavior of the   diagonal cycle for a general Shimura curve predicted by triple product $L$-functions. 
Keep the notations from the last section, except that now we abandon \eqref{Kv} and \eqref{Kv0}. 
We analyze two cases   when the triple product  $L$-function is decomposable. Such  consideration appeared in an earlier version of \cite{YZZ0} for the heights of  the   diagonal cycle.

First, if $ \pi_1\cong \pi_2^\vee=:\pi$, then 
\begin{equation}
  L(s, \pi_1\otimes\pi_2\otimes\pi_3)  = L(s,\ad \pi\otimes \pi_3) L(s,\pi_3).\label{LLL}
\end{equation}
Assume $\pi_3$ is   self-dual
with  root number $\ep(1/2, \pi_3)=-1$. Then $L(s, \pi_1\otimes\pi_2\otimes\pi_3)$  has    vanishing order   at least 3  at $s =1/2$ if and only if 
$L(s,\ad \pi\otimes \pi_3) $
has    vanishing order   at least 2  at $s =1/2$  or  $
L(s,\pi_3)$ has    vanishing order   at least 3  at $s =1/2.$
Then  by \Cref{lemroot},   we have the following.
\begin{prop}\label{rmkdecom}
  In the situation of \Cref{thmmainCM}, 
 if $ \pi_1\cong \pi_2 =:\pi$ (which are self-dual) and $
L'(1/2,\pi_3)\neq 0$, then 
 $L(s,\ad \pi\otimes \pi_3) $
has    vanishing order   at least 2  at $s =1/2$.
\end{prop}

Now, we consider  the diagonal cycle.
Let $\pi_{i,\rf}  = \pi_{A_i,\iota_i}$ for some pair $(A_i,\iota_i)$ as in \Cref{YZZ322} (2). 
 If
$[\Delta_3]_{f_1\otimes f_2\otimes f_3} \neq 0     $ in $(\iota_1\times \iota_2\times \iota_3)_*\Ch_1(A_1\times A_2\times  A_3)$, 
the generalized Gross--Kudla conjecture and the conjectural nondegeneracy of the height pairing imply that
  $L'(1/2, \pi_1\otimes\pi_2\otimes\pi_3) \neq 0$. 
Then the Beilinson--Bloch conjecture predicts that $(\iota_1\times \iota_2\times \iota_3)_*\Ch_1(A_1\times A_2\times  A_3)$   has dimension 1 and is thus generated by 
  $[\Delta_3]_{f_1\otimes f_2\otimes f_3}$. 
Moreover, by \eqref{LLL},   
$L'(1/2,\pi_3)\neq 0$ (assuming $\ep(1/2,\pi_3)=-1$).  
Then by the  general  Gross--Zagier formula  \cite{GZ,YZZ} and \Cref{prop:van1} (which applies by \Cref{rmkMA} (2)), we propose the following  conjecture.
  
\begin{conj}\label{conj:decom2}

Assume 
$ \pi_1\cong \pi_2^\vee $,
and
$\pi_3$ is   self-dual
with     $ \ep(1/2,\pi_3)=-1$. 
Let $(A,\iota)$ corresponding to $\pi_3$ under \Cref{YZZ322} (2).
  Then  there exists  a   zero cycle $Z$ on $X_U$ such that
$[\Delta_3]_{f_1\otimes f_2\otimes f_3}      $ is a multiple of  $[\Delta_2]_{f_1\otimes f_2} \times \pr_{\iota,*}f_{3,*}[Z]$. Moreover,  we can let $Z$  be  a  CM 0-cycle associated to some quadratic CM extension of $F$.
 
\end{conj} 

The relation between modified diagonal cycles and CM cycles  was already explored in \cite{DRS}, as well as used \cite{KLQY,LR}.  Our conjecture gives a 
complete  explanation. 
Alternatively, in  an earlier version of \cite{YZZ0}, the authors proposed to  use 
an ample and symmetric line bundle  associated to the polarization $A_1\to A_2\cong A_1^\vee$ to construct 0-cycles that replace CM cycles in this case. See also \cite[5.3]{Zha101}.

Second, if $ \pi_1\cong \pi_2\cong \pi_3=:\pi$ is self-dual, then $$L(s, \pi_1\otimes\pi_2\otimes\pi_3)  = L(s,\sym^3 \pi) L(s,\pi)^2.$$
So if the root number $\ep(1/2, \pi)=-1$, then $L(s, \pi^{\otimes 3})$ has vanishing order at least 2 at $s =1/2$. 
Thus by \Cref{lemroot} and \Cref{corsign2}, \Cref{thmmain} is trivial if $\pi_1\cong \pi_2\cong \pi_3.$ 

For the diagonal cycle, 
motivated by 
the generalized Gross--Kudla conjecture,
we propose the following conjecture.

\begin{conj}\label{conj:decom1}
  If $ \pi_1\cong \pi_2\cong \pi_3=:\pi$ is self-dual and $\ep(1/2, \pi)=-1$, then $[\Delta_3]_{f }     =0$
for  all $f\in \boxtimes_{i=1}^3\pi_{i}^U$. 
  
\end{conj}

\Cref{prop:van0} (which applies by \Cref{rmkMA} (2))   gives   the following proposition. It contains 
a special case  of \Cref{conj:decom1}.

  \begin{prop}   \label{propfff}    
  If $ \pi_1\cong \pi_2\cong \pi_3=:\pi$,
  for $f\in  \sym^3\pi$,    $[\Delta_3]_{f }     =0$.
     \end{prop} 
   \begin{rmk}\label{prop:decom1}
   
If   $\ep(1/2, \pi)=1$, the vanishing in the proposition follows from \Cref{generalcor} and a result of Prasad
\cite[Theorem 6]{Pra}. This is 
   purely local and  representation-theoretical. 
    \end{rmk}

\section{Hyperelliptic modular curves }\label{sec:genus2}
 Consider the case of modular curves, i.e., let $B=\RM_2(\BQ).$
A Shimura curve $X_U$ with 
  $U$ satisfying  \eqref{Kv0} is the  quotient $X_0^*(N)$ of the classical modular curve $X_0(N)$ by the full Atkin--Lehner group, for  some  $N$ square free. (See \Cref{final remark} (2) for an adelic definition of Atkin--Lehner involutions.)  
  Let $U(N)\subset \PGL_2(\BA_\rf)$ be the corresponding open compact subgroup. 
  
By 
\cite{Has}, the list of square free $N$'s with $X_0^*(N)$ hyperelliptic  consists of the following 39 numbers:
\begin{equation}
\label{eq:squarefree-list}
\begin{aligned}
&67, 73, 85, 93, 103, 106, 107, 115, 122, 129,\\
&133, 134, 146, 154, 158, 161, 165, 166, 167, 170,\\
&177, 186, 191, 205, 206, 209, 213, 215, 221, 230,\\
&255, 266, 285, 286, 287, 299, 330, 357, 390.
\end{aligned}
\end{equation}
Moreover, these $X_0^*(N)$'s all have genus 2.  Thus we can not apply \Cref{thmmain}. Instead, we want to apply \Cref{thmmainCM}. 
We need the following lemma.

\begin{lem}
\label{fevan}

Let $N$ be square free, 
 Assume
 \begin{itemize}
  \item[(1)] there exists a prime factor $p \mid N$ that is inert in $\mathbb{Q}(i)$, and
  \item[(2)] there exists a (possibly equal) prime factor $q \mid N$ that is inert in $\mathbb{Q}(\sqrt{-3})$.
\end{itemize}
For  a cuspidal  automorphic  representation $\pi$
of 
$\GL_{2,\BQ} $   holomorphic of weight $2$, if  $f\neq 0\in \pi^{U(N)} $, 
then  
 $\pr_{f,*}[\fe] = 0$ on   $X_0^*(N)$.

\end{lem}
\begin{proof}
The divisor of elliptic points has a decomposition $\fe = \fe_2+\fe_3$ into a CM divisor with CM by $\mathbb{Q}(i)$ and $\mathbb{Q}(\sqrt{-3})$.
We show $\pr_{f,*}[\fe_2] = \pr_{f,*}[\fe_3] = 0$.  We prove the first equation and the proof of the second is the same. 

Let $E = \mathbb{Q}(i) $ embedded in $B = \RM_2(\BQ)$, uniquely up to conjugation.
Using   the formulation in \cite[1.3.2]{YZZ}, we only need to show  
$\Hom_ {E_p^\times} (\pi_p,\BC) = 0.$ 
 Indeed,
 $\pi_p$  is the
the Steinberg representation twisted by the unique nontrivial unramified quadratic character of $\BQ_p^\times$. 
Let $D_p$ be the  unique division quaternion algebra over $\BQ_p$ and $E_p$ embedded in $D_p$, uniquely up to conjugation.
The  Jacquet--Langlands correspondence of  $\pi_p$ to $D_p^\times$  is the unique nontrivial  (quadratic) character $\chi$ of
  $D_p^\times/\BQ_p^\times O_p^\times\cong \BZ/2\BZ $, where $O_p$ is a  maximal order of $D_p$.  
  Since $\Hom_ {E_p^\times} (\chi,\BC) \neq  0$  by assumption (1), by the Tunnell--Saito dicohotmy  \cite[Theorem 1.3]{YZZ}, 
$\Hom_ {E_p^\times} (\pi_p,\BC) = 0.$  The lemma is proved.
\end{proof}

 In $\mathbb{Q}(i)$, a prime $p$ is inert if and only if 
  $p \equiv 3 \pmod{4}$. In $\mathbb{Q}(\sqrt{-3})$, a prime $q$ is inert if and only if 
  $q \equiv 2 \pmod{3}$.  
  Among the $39$ squarefree $N$ in the list, the following $26$ contain 
at least one prime factor $p \equiv 3 \pmod{4}$ and at least one prime factor 
$q \equiv 2 \pmod{3}$.
\begin{table}[h]
\centering
\caption{Squarefree $N$ in \Cref{eq:squarefree-list}
 with prime factors inert in   $\mathbb{Q}(i)$ and prime factors inert in   $\mathbb{Q}(\sqrt{-3})$.}
\label{tab:inert-primes}
\[
\begin{array}{|c|c|c|}
\hline
N & \text{Inert in } \mathbb{Q}(i) & \text{Inert in } \mathbb{Q}(\sqrt{-3}) \\
\hline
107 & \{107\} & \{107\} \\
115 & \{23\} & \{5,23\} \\
134 & \{67\} & \{2\} \\
154 & \{7,11\} & \{2,11\} \\
158 & \{79\} & \{2\} \\
161 & \{7,23\} & \{23\} \\
165 & \{3,11\} & \{5,11\} \\
166 & \{83\} & \{2,83\} \\
167 & \{167\} & \{167\} \\
177 & \{3,59\} & \{59\} \\
186 & \{3,31\} & \{2\} \\
191 & \{191\} & \{191\} \\
206 & \{103\} & \{2\} \\
209 & \{11,19\} & \{11\} \\
213 & \{3,71\} & \{71\} \\
215 & \{43\} & \{5\} \\
230 & \{23\} & \{2,5,23\} \\
255 & \{3\} & \{5,17\} \\
266 & \{7,19\} & \{2\} \\
285 & \{3,19\} & \{5\} \\
286 & \{11\} & \{2,11\} \\
287 & \{7\} & \{41\} \\
299 & \{23\} & \{23\} \\
330 & \{3,11\} & \{2,5,11\} \\
357 & \{3,7\} & \{17\} \\
390 & \{3\} & \{2,5\} \\
\hline
\end{array}
\]
\end{table}

\begin{thm}\label{thm:ord2}

For $N$ in  \Cref{tab:inert-primes}, let $f,g$ be (the only) two distinct newforms of level $\Gamma_0(N)$ and Atkin--Lehner sign $+1$ 
at all primes dividing $N$.  
Then 
 $$\ord_{s=2}L\lb s, \sym^2 f \otimes  g  \rb\geq  2.$$

\end{thm}
\begin{proof} 
By \Cref{fevan},  \Cref{thmmainCM} applies.  By \cref{rmkdecom}, we only need to show that 
$L'(1,g)\neq 0$. This can be read from LMFDB \cite{LMFDB}.
\end{proof}

\section{An intriguing hyperelliptic Shimura curve of  Demb\'el\'e }\label{Demcurve}

 Let   $F=\BQ(\zeta_{32}+\zeta_{32}^{-1})$, the maximal totally real subfield of the cyclotomic field of 32nd roots of unity.
In this section, we describe 
 a hyperelliptic Shimura curve $X_U$ over $F$  found by Demb\'el\'e \cite[5.6]{Dem}. 
We first clarify that in \cite[5.6]{Dem}, a Shimura curve  is of the form $X_V$ with $V$ open compact (instead of compact-modulo-center as in our paper) in $B^\times(\BA_\rf)$.
 And the  hyperelliptic curve constructed in loc. cit. is an Atkin--Lehner quotient of such a Shimura curve, which is still a Shimura curve in our sense.

Let us be more explicit. 
Let $\tau$ be any of the 8 infinite place of $F$. 
Let $\fp$ be the unique prime of $F$ above 2. To keep notations close to \cite{Dem}, let us use  $D$ to denote the quaternion algebra over $F$ being divison exactly at  $\fp$ and  all  infinite places of $F$ except $\tau$.
 Let $V  =\prod\limits_{v\text{ finite}} V_v \subset D^\times(\BA_\rf)$ be a    maximal open compact subgroup.
In other words,
its   $\fp$-component is 
          \begin{equation*} 
V_\fp= O_{\fp }^\times  ,
 \end{equation*}
 where $O_\fp$ is a  maximal order of $D_\fp$,
 and 
             \begin{equation*} 
V_v=\GL_2\lb \cO_{F_v}\rb  
\text{ for } v\neq\fp.
 \end{equation*}
Since the narrow class number of $F$ is 1, $X_V$ is geometrically connected and $X_V\to X_{V \BA_\rf^\times}$ is an isomorphism. Here 
$\BA_\rf^\times$ is the center of $D^\times(\BA_\rf)$. 
This is the Shimura curve 
$X_0^D(1) $ in  \cite[Introduction, 4.1]{Dem}.
The generator of $D_\fp^\times/F_\fp^\times O_{\fp }^\times\cong \BZ/2\BZ $  normalizes $V_\fp F_\fp^\times$ and thus defines an automorphism of 
$X_{V \BA_\rf^\times}$ of order 2.
(For example we can take the generator to be the uniformizer of a ramified quadratic extension of $F_\fp$ in $D_\fp$.)
This is the Atkin--Lehner involution $w_D$ in  \cite[Introduction, 5.4]{Dem}.
In particular, the Atkin--Lehner quotient
$X_0^D(1)/\pair{w_D}$ in  \cite[Introduction, 5.4]{Dem}
is $X_U$ with 
  \begin{equation*} 
U_\fp= PD_{\fp }^\times  
 \end{equation*}
 and 
             \begin{equation*} 
U_v=F_v^\times \GL_2\lb \cO_{F_v}\rb  
\text{ for } v\neq \fp. 
 \end{equation*}

  \begin{thm}[{\cite[Theorem 5.9]{Dem}}]The curve $X_U$
    is hyperelliptic over F.
    
  \end{thm}
We only need the hyperellipticity over $\BC$.

By \cite[5.3]{Dem}, 
there are 40 holomorphic  Hilbert
 newforms  of level $\Gamma_0(\fp)$ and parallel weight 2, and they form  5 Galois orbits of sizes  4, 4, 4, 4 and 24 respectively.
Moreover, let $w$ be the Atkin--Lehner involution acting on the $\Gamma_0(\fp)$-Hilbert
 modular forms. Then the newforms in the four Galois orbits of sizes 4 are exactly the ones with 
$w$-eigenvalue $-1$.
Their Jacquet--Langlands correspondences to $D^\times$   exactly the ones with 
$w_D$-eigenvalue $1$. 
So these Jacquet--Langlands correspondences  descend to  $X_{U,\BC} $ and span the space of
  holomorphic differential 1-forms on  $X_{U,\BC} $, which has genus 16.
 Then by \Cref{thmmain}, we have the following.

 \begin{thm}\label{exmain}  
Consider    the   16 holomorphic 
  Hilbert
 newforms  of level $\Gamma_0(\fp)$ and parallel weight 2 with $w$-eigenvalue $-1$. Let $f_1 , f_2, f_3$ be any  three distinct ones among them. Then
   $$\ord_{s=2}L\lb s, f_1\otimes  f_2\otimes  f_3 \rb\geq  3.$$

  \end{thm}

 \Cref{thmintro} follows from \Cref{exmain}.


\begin{thebibliography}{99}  
                
                     
 \bibitem{DRS}   Darmon, Henri, Victor Rotger, and Ignacio Sols. \textit{Iterated integrals, diagonal cycles and rational points on elliptic curves.} Publications math\'ematiques de Besancon. Algebre et th\'eorie des nombres 2 (2012): 19-46.                         
 \bibitem{DM}      Deninger, Christopher  and Jacob Murre. \textit{Motivic decomposition of abelian schemes and the Fourier transform.} Journal f\"ur die reine und angewandte Mathematik 422 (1991): 201-219.
                       
                        
                    
                                                               
              
                                    \bibitem{Dem}                       Demb\'el\'e, Lassina. \textit{An intriguing hyperelliptic Shimura curve quotient of genus 16.} Algebra $\&$ Number Theory 14.10 (2020): 2713-2742.
                 \bibitem{Gar}        Garrett, Paul B.  \textit{Decomposition of Eisenstein series: Rankin triple products.}Annals of Mathematics 125.2 (1987): 209-235.

\bibitem{Gold}    Goldfeld, Dorian. \textit{The Gauss class number problem for imaginary quadratic fields.} Heegner points and Rankin L-series 49 (2004): 25-36.
                      
                       \bibitem{GK}  Gross, Benedict H., and Stephen S. Kudla. \textit{Heights and the central critical values of triple product $ L $-functions.}Compositio Mathematica 81.2 (1992): 143-209.
     \bibitem{GKZ}   Gross, Benedict H., Winfried Kohnen, and Don Zagier. \textit{Heegner points and derivatives of L-series. II.} Mathematische Annalen 278 (1987): 497-562.

     \bibitem{GP}Gross, Benedict H., and Dipendra Prasad. \textit{Test vectors for linear forms.} Mathematische Annalen 291.1 (1991): 343-355.
                    \bibitem{GS}      Gross, Benedict H., and Chad Schoen. \textit{The modified diagonal cycle on the triple product of a pointed curve.} Annales de l'institut Fourier. Vol. 45. No. 3. 1995.
                
\bibitem{GZ}Gross, Benedict H., and Don B. Zagier.  \textit{Heegner points and derivatives of $L$-series}. Inventiones mathematicae 84.2 (1986): 225-320. 
\bibitem{Has}Hasegawa, Yuji. \textit{Hyperelliptic modular curves $ X^* _0 (N) $.} Acta Arithmetica 81.4 (1997): 369-385.
\bibitem{KLQY}Kerr, Matt, Wanlin Li, Congling Qiu, and Tonghai Yang. \textit{Non-vanishing of Ceresa and Gross--Kudla--Schoen cycles associated to modular curves.} arXiv preprint arXiv:2407.20998 (2024). 
   \bibitem{Kim}  Kimura, Shun-Ichi.  \textit{Chow groups are finite dimensional, in some sense.} Mathematische Annalen 331 (2005): 173-201.
              \bibitem{LMFDB}LMFDB Collaboration, \textit{The L-functions and modular forms database, http://www.lmfdb.org.}

 \bibitem{LR}    
Lupoian, Elvira, and James Rawson. \textit{Ceresa Cycles of $ X_ {0}(N) $.} arXiv preprint arXiv:2501.14060 (2025).
 
   
         \bibitem{PSR}    Piatetski-Shapiro, Ilya, and Stephen Rallis.  \textit{Rankin triple $L$-functions.} Compositio Mathematica 64.1 (1987): 31-115.
          
    \bibitem{Pra}  Prasad, Dipendra. \textit{Trilinear forms for representations of $\mathrm {GL}(2) $ and local $\epsilon $-factors.} Compositio Mathematica 75.1 (1990): 1-46.
        \bibitem{Pra1}  Prasad, Dipendra. \textit{Relating invariant linear form and local epsilon factors via global methods.} Duke Mathematical Journal 138.2 (2007): 233-261.
 
    \bibitem{Qiufin}  Qiu, Congling. \textit{Finiteness properties for Shimura curves and modified diagonal cycles.} arXiv preprint arXiv:2310.20600 (2023).

    \bibitem{QiuFP}  Qiu, Congling. \textit{Faber--Pandharipande cycle, real multiplication and torsion points.} arXiv preprint arXiv:2409.08989v2 (2024)

      \bibitem{QZ1}
    Qiu, Congling, and Wei Zhang. \textit{Vanishing results in Chow groups for the modified diagonal cycles.} Tunisian Journal of Mathematics 6.2 (2024): 225-247.
 
          \bibitem{QZ2} Qiu, Congling, and Wei Zhang. \textit{Vanishing results in Chow groups for the modified diagonal cycles II: Shimura curves.} Preprint, available at https://arxiv. org/abs/2310.19707 v1 (2023). 
            \bibitem{Schm} Schmidt, Ralf. \textit{Some remarks on local newforms for $\GL (2)$.} JOURNAL-RAMANUJAN MATHEMATICAL SOCIETY 17.2 (2002): 115-147.

        
                  \bibitem{YZZ0}  Yuan, Xinyi, Shouwu Zhang, and Wei Zhang. \textit{Triple product L-series and Gross--Kudla--Schoen cycles.} In progress.
     \bibitem{YZZ}Yuan, Xinyi, Shouwu Zhang, and Wei Zhang. \textit{The Gross-Zagier formula on Shimura curves.} No. 184. Princeton University Press, 2013.
 %
  

       \bibitem{Zag}Zagier, Don.  \textit{L-series of elliptic curves, the Birch-Swinnerton-Dyer conjecture, and the class number problem of Gauss.} Notices Amer. Math. Soc 31.7 (1984): 739-743.
      \bibitem{Zha10}Zhang, Shou-Wu. \textit{Gross--Schoen cycles and dualising sheaves.} Inventiones mathematicae, 1.179 (2010): 1-73.

  \bibitem{Zha101}   Zhang, Shou-Wu. \textit{Arithmetic of Shimura curves.} Science China Mathematics 53.3 (2010): 573-592.
                                                       \end{thebibliography}
\end{document}